\documentclass{article}

\usepackage{amsmath}         
\usepackage{amssymb} 
\usepackage{amsthm}  
\usepackage{enumerate}
\usepackage{stmaryrd}
\usepackage{url} 
\usepackage{mathrsfs}       
\usepackage[all]{xy}

\theoremstyle{plain}
\newtheorem{theorem}{Theorem}[section]
\newtheorem{lemma}[theorem]{Lemma}
\newtheorem{corollary}[theorem]{Corollary}

\theoremstyle{definition}

\numberwithin{equation}{section}

 \let\pL=\L
 
 \def\kn{\kern.1em}

\def\A{\mathbf{A}}
\def\B{\mathbf{B}}
\def\V{\mathbf{V}}

\def\P{\mathcal{P}}

\def\fr{\mathit{Fr}}

\def\cont{\mathfrak{c}}
\def\C{\mathfrak{C}}
\def\F{\mathfrak{F}}
\def\L{\mathfrak{L}}
\def\M{\mathfrak{M}}

\def\Cm{\mathsf{Cm} \kn}

\def\Em{\mathsf{Em} \kn}

\def\Im{\mathsf{I}\mathsf{m}}

\def\Cf{\mathbb{C} }    

\def\lk{{L_\k}}
\def\lo{{L_\om}}

\def\k{\kappa}
\def\om{\omega}
\def\ph{\phi}
\def\sig{\sigma}

\def\sub{\subseteq}
  \def\tprod{{\textstyle\prod}}

\begin{document}

\title{Fine's Theorem on First-Order\\  Complete Modal Logics}
\author{Robert Goldblatt\\
Victoria University of Wellington}
\date{}       
\maketitle

\begin{abstract}
Fine's influential Canonicity Theorem states that if a modal logic is determined by a first-order definable class of Kripke frames, then it is valid in its canonical frames. This article reviews the background and context of this result, and the history of  its impact on further research. It then develops a new characterisation of when a logic is canonically valid, providing a precise point of distinction with the property of first-order completeness. The ultimate point is that the   construction of  the canonical frame of a modal algebra does not commute with the ultrapower construction.
\end{abstract}

\section{The Canonicity Theorem and Its Impact}   \label{canthm}

In his PhD research, completed in 1969, and over the next half-dozen years,  Kit Fine made a series of fundamental contributions to the semantic analysis and metatheory of propositional modal logic,  proving general theorems about notable classes of logics and providing examples of failure of some significant properties. This work  included the following (in order of publication):
\begin{itemize}
\item
A study \cite{fine:prop70} of logics that have propositional quantifiers  and are defined semantically by constraints on the range of interpretation of the quantifiable variables as subsets of a Kripke model. Axiomatisations were given for cases where the range of interpretation is either the definable subsets, or an arbitrary Boolean algebra of subsets. Non-axiomatisability was shown for some cases where the range includes all subsets and the underlying propositional logic is weaker than S5. Decidability and undecidability results were also proved.
\item
A model-theoretic proof \cite{fine:logi71} of Bull's theorem (originally proved algebraically) that all normal extensions of S4.3 have the finite model property. It was also shown that these logics are all finitely axiomatisable and decidable,  and a combinatorial characterisation was given of the lattice of extensions of S4.3 which showed that it is countably infinite.
\item
Construction \cite{fine:logi72} of a  modal logic extending S4, and a superintuitionistic propositional logic,  that are finitely axiomatisable and lack the finite model property. Previously Makinson \cite{maki:norm69} had constructed a sublogic of S4  with these characteristics.
\item
Axiomatisations of logics with `numerical' modalities $M_k$, for positive integer $k$, meaning `in at least $k$ possible worlds' \cite{fine:inso72}. This topic later became known as \emph{graded} modal logic.
\item 
Exhibition of a logic extending S4 that is incomplete for validity in its Kripke frames \cite{fine:inco74}.   This paper and  one of
S.~K.~Thomason  \cite{thom:inco74} independently provided the first examples of incomplete modal logics. Thomason's was a sublogic of S4, following his earlier  discovery of an incomplete tense logic  \cite{thom:sema72}.
\item
A proof \cite{fine:asce74} that the lattice of logics extending S4 is uncountable and includes an isomorphic copy of the powerset 
$(\mathcal{P}(\om),\sub)$ of the natural numbers, ordered by inclusion.
\item
An extensive study \cite{fine:logi74} of the model theory of logics that extend the system K4, i.e.\ their Kripke frames are transitive. This included a proof of Kripke-frame completeness  for any such logic whose frames have a fixed bound on their `width', or degree of branching.   Another aspect of this project that was published later \cite{fine:logi85} focused on the \emph{subframe logics}, each of which is determined by a class of frames that is closed under subframes. There are uncountably many of them, and all were shown to  have the finite model property, and other features that will be mentioned later.
\item
A theory \cite{fine:norm75} of normal forms in modal logic, leading to a proof that all members of a certain class of `uniform' logics have the finite model property, hence are Kripke-complete, and are decidable if finitely axiomatisable. In particular, this included the smallest logic containing the well-known  McKinsey axiom $\Box\Diamond p\to\Diamond\Box p$, for which no such results had been available.

\end{itemize}
My interest here is in what was arguably the most influential contribution: the paper \emph{Some connections between elementary and modal logic} 
\cite{fine:conn75}, and in particular its Theorem 3, which will be referred to as \emph{Fine's Canonicity Theorem.} It states that
\begin{quote}
\em any logic that is complete with respect to a first-order definable class of Kripke frames must be valid in its canonical frames. 
\end{quote}
\noindent
These canonical frames have maximally consistent sets of formulas as their members (`possible worlds'), and their use   is an extension of the famous method of completeness proof introduced by Henkin for first-order logic and the theory of types 
\cite{henk:comp49, henk:comp50, henk:disc96}. Prior to \cite{fine:conn75}, propositional modal logics were typically taken to be based on a denumerably infinite set of variables, but Fine took the step of allowing languages to have arbitrarily large sets of variables, from which arbitrarily large canonical frames can be built for any given logic.

The above body of work by Fine can be seen as part of a second wave of research that flowed from the publication by Kripke \cite{krip:sema63a} of his seminal work on the relational semantics of normal propositional modal logics. As is well known, one reason for the great success of Kripke's theory was that it provided models that were much easier to conceptualize, construct  and manipulate than the algebraic structures that had been used hitherto.
Many logics could be shown to be determined by classes of frames defined by simple \emph{first-order} conditions on a binary relation. Thus S4 is determined by the class of preorders (reflexive transitive relations) and by the class of partial orders (antisymmetric as well); S5 by the class of equivalence relations and by the class of universal relations; S4.3 by the class of linear orders; etc.

The first wave of research focused on exploring this phenomenon for various logics. The Henkin method was applied to relational semantics by a number of people, including Cresswell \cite{cres:henk67}, Makinson \cite{maki:comp66} and Lemmon and Scott \cite{lemm:inte66}. 
The latter defined a particular model  $\M_L$ for any logic $L$, based on the frame  $\F_L$ of all maximally $L$-consistent sets.  This model determines $L$: the formulas true in  $\M_L$ are precisely the $L$-theorems. Therefore, if $\F_L$ satisfies a certain condition $C$ that makes $L$ valid, i.e. makes the $L$-theorems true in all models on the frame, it follows that $L$ is not only determined by $\F_L$, but is also determined by the class of all frames satisfying $C$.
For example, the axioms of S4 ensure that the frame $\F_\mathrm{S4}$ is a preorder, and that suffices to imply that S4 is determined by validity in all preorders.

The adjective \emph{canonical} was attached to the model $\M_L$  by Segerberg 
\cite{sege:deci68,sege:essa71}.   Its underlying frame $\F_L$ is the \emph{canonical frame} of $L$, and
 $L$  itself may be called a \emph{canonical} logic if it is valid in $\F_L$.\footnote{This use of `canonical' will be refined later. See footnote \ref{noteLom}.}
By the end of the 1960's, numerous logics had been shown to be canonical, and hence complete for their Kripke-frame semantics, by showing that $\F_L$ satisfies some first-order definable conditions that validate $L$. Such a proof also shows that $L$ is \emph{first-order complete}, in the sense that it is complete for validity in some first-order definable class of frames.

The second wave moved beyond the characterisation of particular systems to consider metalogical questions about the nature of Kripke semantics itself and its relation to algebraic semantics and to other logical formalisms, notably first-order and monadic second-order predicate logic. Other contributors included S.~K.~Thomason, Gabbay, Esakia, van Benthem, Blok and myself.\footnote{See \S 6 of \cite{gold:math06} for a survey of this metatheory from the 1970's.} The earlier work had suggested a close connection between propositional modal logic and first-order logic, but it was recognised that the notion of validity in a frame is intrinsically second-order, since it refers to truth in all models on a frame, hence has the effect of allowing  propositional variables to range in value over arbitrary subsets of the frame. Thomason \cite{thom:redu75seco} gave a reduction of the full monadic second-order theory of a binary relation to the propositional logic of a single modality.\footnote{See \cite[\S 6.4]{gold:math06}  for a summary.}

Every logic that had been shown to be canonical had been done so in a way, explained two paragraphs ago, that also showed it to be first-order complete. Moreover, the only known examples of non-canonical logics had an axiom expressing a non-first-order property of frames. For example, Fine proved in \cite[page 38]{fine:logi74} the non-canonicity of the extension of S4.3 by the axiom
$$
\neg[p\land  \Box(p\to \Diamond (\neg p\land\Diamond p))].
$$
This axiom  is deductively equivalent to the more well known formula
$$
\Box(\Box(p\to\Box p)\to p)\to p, \footnote{Due to Soboci{\'n}ski and commonly known as Grz:  see \cite[pp.~168--169]{sege:essa71} for the origin of this terminology.}
$$
and is valid in a linearly ordered frame iff the ordering has no infinite strictly increasing sequence, a non-first-order condition. A weaker formula is
Dummett's Diodorean axiom
$$
\Box(\Box(p\to\Box p)\to\Box p)\to(\Diamond\Box p\to\Box p),
$$
which corresponds in  the ordering of the natural numbers to the non-first-order discreteness property  that between any two points there are only finitely many other points. Earlier, Kripke
\cite{krip:revi67}  had observed that Dummet's formula
is not preserved by the J\'onsson--Tarski representation of modal algebras. This is an algebraic formulation of the non-canonicity of this formula.

The appearance of Fine's Canonicity Theorem gave a theoretical explanation of the observed situation, since it showed that
\begin{quote}
\em every  first-order complete logic is canonical.  
\end{quote}
\noindent
Thus any non-canonical logic must fail to be determined by any first-order definable class of frames.

There was another result about canonicity in  \cite{fine:conn75}. Theorem 2 showed that if the class $\fr(L)$ of all frames that validate $L$ is closed under first-order equivalence of frames, and $L$ is complete for validity in $\fr(L)$, then $L$ is valid in its canonical frames. But it turned out that this result follows from the Canonicity Theorem 3. Johan van Benthem  made the arresting discovery that if  the class $\fr(A)$ of frames validating a single modal formula $A$ is closed under first-order equivalence, then has the ostensibly stronger property of being defined by a single first-order sentence. That prompted me to observe that  if a class of frames is closed under ultrapowers and certain modal-validity preserving operations, then it must be closed under ultraproducts. From this it can be shown, for any logic $L$, that if
$\fr(L)$ is closed under first-order equivalence, then it is definable by some set of first-order sentences.\footnote{van Benthem's original proof involved the use of first-order compactness, but in the published version \cite{bent:moda76form} he chose to give my structural argument about closure under ultraproducts. In \cite{gold:refl99} I was able to return the favour by publishing an account of his original proof.}
Therefore the hypotheses of Theorem 2 imply the hypothesis of Theorem 3 that $L$ is first-order complete.

Segerberg's review of  \cite{fine:conn75} in \cite{sege:revi} suggested that `the paper may well open up a new avenue of research in modal logic'. Indeed it did.
Its most significant innovation  was the use of Kripke models that are \emph{saturated} \cite[Chapter 5]{chan:mode73} when viewed as models for a first-order language having a binary predicate symbol and a set of monadic ones. Fine showed that if a model $\M$ for this language is sufficiently saturated, then it can be mapped by a modal-validity preserving transformation (\emph{p-morphism}, or\emph{ bounded morphism}) onto the canonical model of the modal theory determined by $\M$.

I was fortunate to have access to a preprint of \cite{fine:conn75} in late 1973 as I was completing a PhD thesis on the duality between Kripke models and modal algebras \cite{gold:meta74}. I soon saw that saturation of models could be applied to derive a theorem giving structural criteria for when a first-order definable class of frames is \emph{modally} definable, i.e.\ is the class of all frames that validate some set of modal formulas. This involved more  validity preserving constructions: \emph{inner} (or \emph{generated}) \emph{subframes} and \emph{disjoint unions}, as well as the \emph{ultrafilter extension}  of a frame 
$\F$. The latter is a frame $\mathfrak{ue\,F}$ whose points are the ultrafilters on the underlying set of $\F$, and is an analogue of the notion of canonical frame.  One version of my theorem states that:
\begin{equation}\label{gtthm}
\parbox{.8\linewidth}{ \em
A first-order definable class $\C$ of frames is modally definable if, and only if, $\C$ is closed under disjoint unions, images of bounded morphisms and inner subframes, and its complement is closed under ultrafilter extensions.}
\end{equation}
The motivation for such a result was the desire to understand more precisely the relationship between the modal and first-order formalisms. This led to the
question of which first-order definable properties of a binary relation are expressible by modal formulas, given the observed converse      phenomenon that many, but not all,  modal-definable classes were first-order definable. \eqref{gtthm} is reminiscent of the celebrated theorem of Birkhoff \cite{birk:stru35} that a class of algebras is definable by equations iff it is closed under homomorphic images, subalgebras and direct products. Indeed the proof of \eqref{gtthm} involved an application of Birkhoff's theorem via the duality theory.
The hypothesis of first-order definability  ensured that $\C$ itself is closed under ultrafilter extensions, because 
$\mathfrak{ue\,F}$ is the image of a bounded morphism defined on  a suitably saturated elementary extension $\F'$ of 
$\F$ -- a structural version of Fine's use of saturated models.

Now the extension $\F'$ here satisfies the same first-order sentences as $\F$, so
the assumption on $\C$ can be weakened to closure under first-order equivalence. In this form \eqref{gtthm} appeared in \cite{gold:axio75} and in the published version \cite{gold:meta76} of my thesis. It has now become known as the Goldblatt--Thomason Theorem (see \cite[p.~186]{blac:moda01} for background), and versions of it have been developed for other formalisms, including
hybrid languages \cite{cate:mode05},
graded modal languages \cite{sano:gold10}, and the logic of coalgebras \cite{kurz:gold07}.
A suitably saturated elementary extension of $\F$ can always be obtained as an \emph{ultrapower} of $\F$
 \cite[\S 6.1]{chan:mode73}, and so the hypothesis on $\C$ can be stated as closure under ultrapowers.
Ultimately it can be weakened to just require closure of $\C$ under ultrafilter extensions, as in \cite{kurz:gold07}, although doing  so severs the link with first-order logic and the original motivation for the theorem.

Saturation later came to play an important role in the study of \emph{bisimulation} relations between modal models. These relations generalise the properties of p-morphisms, and were introduced as \emph{p-relations} by van Benthem \cite{bent:moda76corr, bent:moda83}, who showed that the propositional modal language can be identified with the bisimulation--invariant fragment of the associated first-order language. This can be efficiently proved using saturated models (see \cite[p.~120]{blac:moda01} for the history).

Fine's Canonicity Theorem itself stimulated further research, including a reformulation of it as a result about \emph{varieties}, i.e.\ equationally-definable classes of algebras. Each modal algebra $\A$ has a \emph{canonical frame} $\Cf\A$ whose members are the ultrafilters of $\A$. The ultrafilter extension $\mathfrak{ue\,F}$ of a frame is just the canonical frame of the \emph{powerset algebra} of $\F$, i.e.\ the algebra of all subsets of $\F$. For each algebra $\A$, the powerset algebra of $\Cf\A$ itself is called the \emph{canonical embedding algebra} of $\A$, denoted $\Em\A$ (these notions are all fully explained in Section \ref{varieties} below). A \emph{canonical variety} is an equationally definable class $\V$ of algebras that is closed under these canonical embedding algebras, i.e.\ has $\Em\A$ in $\V$ whenever $\A$ is in $\V$.

Now for a class $\C$ of frames, let $\V_\C$ be the variety consisting of the modal algebras that satisfy those equations that are satisfied by the powerset algebras of all the frames in $\C$. $\V_\C$ is called  the variety \emph{generated by} $\C$. When $\C=\fr(L)$  for some logic $L$, then $\V_\C$ is a canonical variety if, and only if, $L$ is a canonical logic.
Fine's Theorem thus suggests the following algebraic assertion:
\begin{equation}  \label{VCthm}
\textit{if $\C$ is first-order definable, then $\V_\C$ is a canonical variety.}
\end{equation}
I gave a proof of this in \cite[3.6.7]{gold:vari89} that used results from universal algebra about the existence of subdirectly irreducible algebras in generated varieties. The proof applies when $\C$ is any class of similar relational structures, having any kind and number of finitary relations. In general a member of $\V_\C$ is then a \emph{Boolean algebra with operators}, or BAO, a notion introduced by 
J{\'o}nsson and Tarski \cite{jons:bool48,jons:bool51}, who first showed how $n+1$-ary relations on a set $X$ correspond to $n$-ary operations on the powerset algebra of $X$ (a Kripke frame falls under the special case $n=1$ of this). BAO's include many kinds of algebra that have been studied by algebraic logicians, including cylindric algebras \cite{henk:cyli71}, polyadic algebras \cite{halm:alge62},  relation algebras  \cite{jons:bool52, hirs:rela02}, and algebras for temporal logic, dynamic logic and other kinds of multi-modal logic.

In \cite{gold:clos91}  I gave a quite different second proof of \eqref{VCthm} that analysed the way in which the canonical structure $\Cf\A$ of each $\A$ in $\V_\C$ could be constructed from members of $\C$ by forming images of bounded morphisms, inner subframes and disjoint unions. This was studied further in \cite{gold:elem95}, leading to strengthenings of  \eqref{VCthm}  that gave information about other first-order definable classes that
generate $\V_\C$. In particular, if $\C$ is first-order definable, then:
\begin{itemize}\em
\item 
$\V_\C$ is also generated by some first-order definable class that includes the class $\{\Cf\A: \A\in\V_\C\}$ of canonical structures of members of $\V_\C$.
\item
$\V_\C$ is also generated by the  class of structures that satisfy the same first-order sentences as the canonical structure\/ $\Cf\A_{\V_\C}$, where 
$\A_{\V_\C}$ is the free algebra in $\V_\C$ on denumerably many generators.
\end{itemize}
This last result has an interesting interpretation when $\C$ is the class $\fr(L)$ of all $L$-frames for some modal logic $L$. Then $\A_{\V_\C}$ can be constructed as the Lindenbaum-Tarski algebra of $L$, and $\Cf\A_{\V_\C}$ can be identified with the canonical $L$-frame $\F_L$. The result implies the following:
\begin{quote}\em
if $ L$ is first-order complete, then it is determined by the class of frames that satisfy the same first-order sentences as does its canonical frame $\F_L$.
\end{quote}
In other words: if $L$ is determined by some first-order conditions, then it is determined by the first-order conditions that are satisfied by the canonical frame $\F_L$.

This conclusion was further refined in \cite{gold:elem93} and \cite{gold:quas01} by studying certain `quasi-modal'  sentences which are defined syntactically, and are the \emph{first-order} sentences whose truth is preserved by the key modal-validity preserving constructions of  bounded morphisms, inner subframes and disjoint unions. It turns out that  if $L$ is a first-order complete modal logic, then any frame that satisfies the same quasi-modal first-order sentences as $\F_L$ must validate $L$, and so $L$ is determined by the class of such frames. This is explained more fully at the end of Section \ref{strength}.

At the end of \cite{fine:conn75}, Fine asked two pertinent questions. One was whether the converse of the Canonicity Theorem is true: must a canonical logic be first-order complete? That was a very natural question, given that, as noted earlier,  all logics that had been shown to be canonical had been done so in a way that showed them to be first-order complete. We can formulate the question more generally as: must a canonical variety of BAO's be generated by some first-order definable class of relational structures? Over the years a considerable number of partial positive answers were found, showing that the answer is yes for various families of logics and varieties. A list of these appears in \cite[pages 189--190]{gold:erdo04}. It includes the \emph{subframe logics}, introduced by Fine in \cite{fine:logi85} where he showed that every canonical subframe logic having transitive frames must be first-order complete. Wolter \cite{wolt:stru97} later removing this transitivity restriction. But eventually, after three decades, it was found that the answer to Fine's converse question is negative in general. Uncountably many counter-examples were given in \cite{gold:erdo04} and \cite{gold:cano04}. So canonicity is not equivalent to first-order completeness. The main function of this paper is to give a structural analysis that accounts for their difference.

The other question raised at the end of \cite{fine:conn75}  was whether a logic that is validated by its canonical frame built from a countable language must be validated by its canonical frames built from larger languages. That remains an open problem.

\bigskip
The rest of this paper will set out the background to these ideas and discoveries in more detail, in order to present some new ones.  Separate characterisations will be given of the notions `first-order complete' and `canonical'  that gives a precise point of distinction between them (see \eqref{focchar} and \eqref{conchar}). If $\A$ is the denumerably-generated Lindenbaum-Tarski algebra of logic $L$, the distinction is between the canonical frame $\Cf(\A^{U})$ of an ultrapower $\A^{U}$ of $\A$, and the corresponding ultrapower $(\Cf\A)^U$ of the canonical frame of $\A$.  In forming ultrapowers and canonical frames, the order of formation matters. $\Cf(\A^{U})$ may be bigger than $(\Cf\A)^U$ (see Theorem \ref{bigger}). Moreover, although $(\Cf\A)^U$ is isomorphic to a subframe of $\Cf(\A^{U})$, this is not in general an \emph{inner} subframe,  and $L$ can be valid in $\Cf(\A^{U})$ but falsifiable in $(\Cf\A)^U$. 

Our principal result, providing the point of distinction, is that $L$ is first-order complete when
\emph{every ultrapower $(\Cf\A)^U$ validates $L$,}
whereas $L$ is canonical when
\emph{every frame $\Cf(\A^{U})$ validates $L$.}

The paper concludes by applying these observations to a simple proof of the converse of the Canonicity Theorem for subframe logics.

\section{Background}

This section briefly reviews what we need from the relational model theory of propositional modal logic (see \cite{blac:moda01} for more details).

From a given class of variables $p_\xi$, one for each ordinal $\xi$, \emph{formulas} are generated using Boolean connectives ($\bot$ and $\to$) and the modality $\Box$. The connectives $\neg$, $\land$, $\lor$, $\leftrightarrow$, 
$\Diamond$ are defined in the usual way. A \emph{logic} is any class $L$ of formulas that contains all instances of tautologies and of the scheme
$$
\Box(A\to B)\to(\Box A\to\Box B),
$$
and is closed under the rules of modus ponens, uniform substitution of formulas for variables, and necessitation (from $A$ infer $\Box A$). The members of $L$ are often called its \emph{theorems}.

There are in fact only $2^\om$ logics, since each is determined by its restriction to formulas whose variables belong to $\L_\om=\{p_\xi:\xi<\om\}$.   To explain this: for each infinite cardinal $\k$, define a \emph{$\k$-formula} to be  any formula whose variables belong to the set $\L_\k=\{p_\xi:\xi<\k\}$.  Then the $\k$-formulas form a \emph{set} of size $\k$.
A \emph{$\k$-logic} is any set of $\k$-formulas that fulfils the definition of a logic as above when restricted to $\k$-formulas, in particular being closed under uniform substitution of $\k$-formulas for variables. For any logic $L$, the set $L_\k$ of all 
$\k$-formulas that belong to $L$ is a $\k$-logic. Since $\L_\k$ is infinite, we can associate with any formula $A$ a substitution instance of it that is a $\k$-formula, and is an $L$-theorem iff $A$ is. This implies that $L$ is the closure of $L_\k$ under substitution, and is the only logic whose restriction to $\k$-formulas is equal to $L_\k$, i.e.\ for any logic $L'$ we have  $L_\k=L'_\k$ iff $L=L'$.
Similarly, $L_\k$ is the closure of $L_\om$ under substitution of $\k$-formulas, and is the only $\k$-logic whose restriction to $\om$-formulas is equal to $L_\om$.  Thus a logic $L$ is stratified into the increasing sequence 
$$
\{L_\k: \k \text{ is an infinite cardinal}\},
$$
of restricted logics, which is completely determined by $L_\om$ in the way described.

A \emph{frame} $\F=(X,R)$ consists of a binary relation $R$ on a set $X$. A \emph{$\k$-model} $\M=(X,R,V)$ on a frame is given by a \emph{valuation} function $V$ assigning to each variable $p\in\L_\k$ a set $V(p)\sub X$. The truth-relation $\M,a\models A$ (read `$A$ is true at $a$ in $\M$') is defined for all $a\in X$ by induction on the formation of $\k$-formulas $A$, with $\M,a\models p$ iff $a\in V(p)$; the Boolean connectives treated as expected; and
\begin{center}
$\M,a\models \Box A$ \enspace iff \enspace for all $b$ such that $aRb$, \,$\M,b\models A$.
\end{center}
A $\k$-formula $A$ is \emph{true in} $\M$, written $\M\models A$,  if $\M,a\models A$ for all $a\in X$. This property of $A$ depends only on the valuations $V(p)$ of the variables that occur in $A$.

An arbitrary formula $A$ is \emph{valid in} a frame $\F$, written $\F\models A$, if it is true in every $\k$-model on $\F$, for any $\k$ such that $A$ is a $\k$-formula. $\F$ \emph{validates} a class $\Delta$ of formulas, written $\F\models\Delta$,  if every member of $\Delta$ is valid in $\F$. $\F$ is an \emph{$L$-frame} if it validates the logic $L$. $\fr(L)$ is the class $\{\F:\F\models L\}$ of all $L$-frames.

For any frame $\F$, the class $L_\F=\{A:\F\models A\}$ of all formulas valid in $\F$ is a logic. If $\C$ is a class of frames, then the class $L_\C$ of formulas valid in all members of $\C$ is a logic, called the logic \emph{determined by} $\C$. A logic is \emph{complete} if it is determined by some class of frames, i.e.\ if it is equal to $L_\C$ for some class $\C$.

Note that if $\M$ is a $\k$-model, the set $\{A:\M\models A\}$ of $\k$-formulas true in $\M$ need not be a $\k$-logic, since it need not be closed under substitution. This explains the importance of the notion of frame-validity in modal logic.

A frame is also a model for the first-order language of a single binary predicate symbol $\bar R$. A class of frames is \emph{elementary} if it is first-order definable, i.e.\ if it is the class of all models of some set of sentences of this language. A logic $L$ will be called \emph{first-order complete} if it is determined by some elementary class of frames.\footnote{`First-order complete' is the same notion as `quasi-$\Delta$-elementary and complete' of \cite{fine:conn75}. It is also known as `elementarily determined'.}

Propositional variables can be regarded as monadic predicate symbols, so a modal $\k$-model is also a model for the first-order language with signature $\L_\k\cup\{\bar R\}$. Then the definition of the truth relation of $\M$ gives rise to a translation assigning to each modal $\k$-formula $A$ a first-order formula $A'$ with a single free variable, such that for each 
$a$ in $\M$,
\begin{equation}  \label{1storderdef}
\M,a\models A \quad\text{iff}\quad \M\models A'[a],
\end{equation}
i.e.\ $A$ is true at $a$ in $\M$ iff $A'$ is satisfied in the first-order model $\M$ when its free variable is assigned the value 
$a$ (see \cite[Theorem 1]{fine:conn75} or \cite[\S 2.4]{blac:moda01}).

The \emph{canonical $\k$-model} of a logic $L$ is the structure
$$
\M_{L_\k}=(X_\lk,R_\lk,V_\lk),
$$
where $X_\lk$ is the set of all maximally $\lk$-consistent sets of $\k$-formulas; $aR_\lk b$ iff $\{A:\Box A\in a\}\sub b$; and $V_\lk(p)=\{a\in X_\lk:p\in a\}$.  The \emph{$\k$-canonical frame} of $L$ is $\F_\lk=(X_\lk,R_\lk)$.

The canonical $\k$-model satisfies the \emph{Truth Lemma}
$$
\M_\lk,a\models A \quad\text{iff}\quad A\in a,
$$
for all $\k$-formulas $A$ and $a\in X_\lk$. This implies that $\M_\lk\models A$ iff $A$ is an $\lk$-theorem, so $\M_\lk$ is a determining model for $\lk$. From this it follows that for any formula $A$,  $\F_\lk\models A$ implies $A\in L$, i.e.\ $L$ is complete for validity in $\F_\lk$. But it may not be \emph{sound} for this validity: there may be $L$-theorems that are not valid in $\F_\lk$, i.e.\ $\F_\lk$ may not be an $L$-frame. The question of whether/when it is an $L$-frame is thus of some interest, and indeed is the leitmotif of this paper.
So we will say that $L$ is \emph{$\k$-canonical} if $\F_\lk$ is an $L$-frame, and is 
\emph{canonical} if it is $\k$-canonical for all infinite cardinals $\k$.\footnote{The frame $\F_L$ of Section \ref{canthm} is in fact $\F_{L_\om}$, and the historical discussion of that Section refers to $\om$-canonicity.\label{noteLom}}
The Canonicity Theorem then states that every first-order complete logic is canonical in the sense of being validated by its $\k$-canonical frame for all $\k$.

A frame $\F=(X,R)$ is a \emph{subframe} of frame $\F'=(X',R')$ if $X$ is a subset of $X'$ and $R$ is the restriction of $R'$ to $X$. Then $\F$ is an \emph{inner subframe}\footnote{Also called a `generated subframe'.} of $\F'$ if also $X$ is $R'$-closed in the sense that if $a\in X$ and $aR'b\in X'$, then $b\in X$. Inner subframes preserve validity: $\F'\models A$ implies $\F\models A$. They also have the following property:
\begin{equation}  \label{inner}
\parbox{.8\linewidth}{ \em
$\F$ is an $L$-frame if for each $a$ in $\F$ there is an inner subframe of $\F$ that contains $a$ and is an $L$-frame.
 }
\end{equation}

A \emph{bounded morphism}\footnote{Also called a `p-morphism'.}
from $\F=(X,R)$ to $\F'=(X',R')$ is a function $f:X\to X'$ such that $aRb$ implies $f(a)R'f(b)$, and $f(a)R'c\in X'$ implies $c=f(b)$ for some $b\in X$ with $aRb$. Then $f(\F)=(f(X),R'{\restriction} f)$ is an inner subframe of $\F'$, where $f(X)$ is the $f$-image of $X$ and  $R'{\restriction} f$ is the restriction of $R'$ to $f(X)$. Images of bounded morphisms preserve validity: if $\F\models A$ then $f(\F)\models A$. Thus if there exists a \emph{surjective} bounded morphism $f:\F\to\F'$, then $\F\models A$ implies  $\F'\models A$.

The third major modal-validity preserving operation is the \emph{disjoint union} $\coprod_I\F_i$ of a set $\{\F_i:i\in I\}$ of frames. This can be defined as the union of a set of pairwise disjoint isomorphic copies of the $\F_i$'s. Each frame $\F_i$ is isomorphic to an inner subframe of $\coprod_I\F_i$. We have 
$\coprod_I\F_i\models A$ iff for all $i\in I$, $\F_i\models A$.

Next to be reviewed is the definition of the \emph{ultraproduct}    
$$
\tprod_U\F_i =(\tprod_U X_i, R_U)
$$ 
of a set $\{\F_i=(X_i,R_i):i\in I\}$ of frames, modulo an ultrafilter $U$ of subsets of $I$.
$\tprod_U X_i$ is the set of equivalence classes of the Cartesian product set $\tprod_I X_i$ under the equivalence relation $f\sim_U g$ that holds iff
$\{i\in I:f(i)=g(i)\}\in U$. Writing $f_U$ for the equivalence class $\{g:f\sim_U g\}$ of $f$, the relation $f_UR_Ug_U$ is defined to hold iff $\{i\in I:f(i)R_ig(i)\}\in U$.

Given a $\k$-model $\M_i=(\F_i,V_i)$ on each frame $\F_i$, a valuation $V$ on the ultraproduct is defined by putting
$f_U\in V(p)$ iff $\{i\in I:f(i)\in V_i(p)\}\in U$. In the resulting model $\prod_U\M_i$, for all $\k$-formulas $A$, 
$$
\tprod_U\M_i,f_U\models A \quad\text{iff}\quad \{i\in I:\M_i,f(i)\models A\}\in U.     
$$

This follows via the corresponding result for first-order logic (\pL o\'{s}'s Theorem), and the relationship \eqref{1storderdef}. From this it can be shown that 
$$
\tprod_U\M_i\models A \quad\text{iff}\quad \{i\in I:\M_i\models A\}\in U,
$$
and that
\begin{equation} \label{rootpres}
\tprod_U\F_i\models A \quad\text{implies}\quad \{i\in I:\F_i\models A\}\in U
\end{equation}
(cf.~\cite{gold:firs75}).

Now when the frames $\F_i$ for all $i\in I$ are equal to a single frame $\F$, then the ultraproduct is called the \emph{ultrapower} of $\F$ modulo $U$, denoted $\F^U$, and $\F$ is the \emph{ultraroot} of $\F^U$. Then from \eqref{rootpres},  
$$
\F^U\models A \quad\text{implies}\quad \F\models A,
$$
so \emph{ultraroots preserve modal validity}.

Two frames $\F_1$ and $\F_2$ are \emph{elementarily equivalent}, written $\F_1\equiv\F_2$, if they satisfy the same first-order sentences. The profound
\emph{ Keisler--Shelah Ultrapower Theorem} \cite[6.1.15]{chan:mode73} states that elementarily equivalent structures have isomorphic ultrapowers, so if $\F_1\equiv\F_2$, then there exists some ultrafilter $U$ such that $\F_1^U$ is isomorphic to $\F_2^U$.

Let $El(\C)$ be the smallest elementary class including a class $\C$. $El(\C)$  is the class of all models of the first-order sentences that are true of all members of $\C$. It can be shown that if $\C$ is closed under ultraproducts, then $El(\C)$ is the closure
$$
\{\F: \F\equiv \F' \text{ for some }\F'\in \C\},
$$
 of $\C$ under elementary equivalence (see  \cite[\S 4.1]{chan:mode73} or \cite[\S 7.3]{bell:mode69}).
 
 \begin{theorem}\label{ElC}
If a class $\C$ of frames is closed under ultraproducts, then $\C$ and the elementary class $El(\C)$ determine the same modal logic.
\end{theorem}
 \begin{proof}
Let $A$ be a modal formula that is valid all members of $\C$. If $\F_1\in El(\C)$, then from above  $\F_1\equiv\F_2$ for some $\F_2\in\C$. By the Keisler--Shelah Theorem,  there is an ultrafilter $U$ such that $\F_1^U$ is isomorphic to $\F_2^U$. Then $\F_2^U\in\C$ by closure of $\C$ under ultraproducts, so $\F_2^U\models A$. Since modal validity is preserved by isomorphism and ultraroots, this gives $\F_1^U\models A$ and then $\F_1\models A$.

Thus a modal formula valid in all members of $\C$ is valid in all members of $El(\C)$. The converse is immediate as $\C\sub El\C$. Hence $L_\C=L_{El(\C)}$.
\end{proof}

\section{Strengthening the Canonicity Theorem}  \label{strength}

The core of Fine's proof of his Canonicity Theorem is the following fact:

\begin{lemma} \label{fineslemma}
If a logic $L$ is determined by an elementary class $\C$, then for all infinite $\k$ and any member $a$ of $\F_\lk$, there is a frame $\F_a\in\C$ and a bounded morphism $f_a:\F_a\to\F_\lk$ whose image $f_a(\F_a)$ contains $a$.
\qed
\end{lemma}
The desired conclusion that $\F_\lk$ is an $L$-frame follows directly from this, because   $f_a(\F)$ is an inner subframe of $\F_\lk$, and since $\F_a$ is an $L$-frame (being a member of $\C$) and images of bounded morphisms preserve validity, it follows that $f_a(\F)$ is an $L$-frame. Thus every member of $\F_\lk$ belongs to an inner subframe that is an $L$-frame, which implies that $\F_\lk\models L$ by \eqref{inner}.

In the proof of Lemma \ref{fineslemma}, a model $\M$ is constructed as an ultraproduct of models that are based on frames from $\C$, and then an arbitrary $\om$-saturated elementary extension $\M'$ of $\M$ is taken, with $\F_a$ being defined as the underlying frame of $\M'$. Now the frame of $\M$ is in $\C$, as $\C$ is closed under ultraproducts, and a suitable $\om$-saturated $\M'$ can always be realised as an ultrapower of $\M$ \cite[\S 6.1]{chan:mode73}, so the essential property of $\C$ required for the Lemma is that it be closed under ultraproducts. In other words, this reasoning shows that
\begin{quote}\em
if a class $\C$ of frames is closed under ultraproducts, then the logic $L_\C$ that  it determines is canonical.
\end{quote}
However this is not a genuine strengthening of the Canonicity Theorem, despite the \emph{apparently} weaker hypothesis on $\C$, because if $\C$ is closed under ultraproducts, then as shown in Theorem \ref{ElC}, $L_\C$ is also determined by the elementary class $El(\C)$, and so $L_\C$ is first-order complete.

The genuine strengthening we give below is obtained by deriving a stronger conclusion, rather than using a weaker hypothesis.
It is based on the following natural ultrapower generalisation of Lemma \ref{fineslemma}.

\begin{lemma}
If a logic $L$ is determined by an elementary class $\C$, then for all infinite $\k$, any ultrafilter $U$, and any member $a$ of the ultrapower $\F_\lk^{\ \,U}$, there is a frame $\F_a\in\C$ and a bounded morphism $f_a:\F_a\to\F_\lk^{\ \, U}$ whose image $f_a(\F_a)$ contains $a$.
\end{lemma}
\begin{proof}
Let $U$ be an ultrafilter on $I$, and $a=g_U$ for some $g\in X_\lk^{\ \, I}$. For each $i\in I$, $g(i)$ is a member of $\F_\lk$, so by Lemma  \ref{fineslemma} there is a frame $\F_i\in \C$ and a bounded morphism $f_i:\F_i\to\F_\lk$ such that $f_i(\F_i)$ contains $g(i)$.

Let $\F_a$ be the ultraproduct $\Pi_U\F_i$, which belongs to $\C$ as elementary classes are closed under ultraproducts.
Define $f_a:\Pi_U X_i\to  X_\lk^{\ \, U}$ by putting, for each $h\in \Pi_I X_i$,
$$
f_a(h_U)=\langle f_i(h(i)): i\in I\rangle_U.
$$
Routine ultraproduct analysis shows that $f_a$ is a well-defined bounded morphism.
Now for each $i\in I$, $g(i)=f_i(b_i)$ for some $b_i\in X_i$. Define $h$ by $h(i)=b_i$ for all $i$. Then $f_i(h(i))=g(i)$, so
$$
f_a(h_U)=\langle g(i): i\in I\rangle_U=a,
$$
hence $a$ belongs to the image of $f_a$ as required.
\end{proof}
Applying \eqref{inner} to this result, in the way explained above, gives

\begin{corollary}  \label{strongfine}
If $L$ is first-order complete, then every ultrapower of $\F_\lk$ is an $L$-frame.
\qed
\end{corollary}
The conclusion of this statement is now so strong that it is equivalent to the hypothesis:

\begin{theorem} \label{firstchar}
For any logic $L$,  if  $\k$ is any infinite cardinal  the following are equivalent.
\begin{enumerate}[\rm(1)]
\item 
Every ultrapower of $\F_\lk$ is an $L$-frame.
\item
$L$ is determined by the elementary class  $El(\F_\lk)=\{\F: \F\equiv\F_\lk\}$.
\item
$L$ is first-order complete.
\end{enumerate}
\end{theorem}
\begin{proof}
(1) implies (2):  Let $\F\equiv\F_\lk$. By the Keisler-Shelah Theorem, there is an ultrafilter $U$ with $\F^U$ isomorphic to
$\F_\lk^{\ \,U}$. By (1), $\F_\lk^{\ \,U}\models L$. But frame validity is preserved by isomorphism and utraroots, so $\F^U\models L$ and hence $\F\models L$. This shows that every $L$-theorem is valid in all members of  $El(\F_\lk)$. Conversely, a formula valid in $El(\F_\lk)$ is valid in $\F_\lk$, and hence is an $L$-theorem. Therefore (2) holds.

(2) implies (3):  by definition of `first-order complete'.

(3) implies (1): By Corollary \ref{strongfine}.
\end{proof}

For a given $L$, this Theorem holds for each $\k$ separately, so it is natural to wonder if the elementary classes $El(\F_\lk)$ are all the same. This amounts to asking if all of the canonical frames $\F_\lk$ are elementarily equivalent, and thus whether
 $\F_\lk\equiv\F_{L_\om}$ for all infinite $\k$.  The answer is not known.
 
 But something is known about equivalence of canonical frames relative to a restricted form of sentence in the first-order language of frames. Define a sentence to be \emph{quasi--modal} if it has the form $\forall x\rho$, with $\rho$ being a formula that is constructed from
atomic formulas by using only the connectives $\land$ (conjunction), $\lor$
(disjunction), and the {\em bounded\/} universal and existential quantifier forms
$\forall z(yRz\to\tau)$ and $\exists z(yRz\land\tau)$ with $y\ne z$. 
Any first-order sentence whose truth is preserved by inner subframes,  images of bounded morphisms and disjoint unions is logically equivalent to a quasi--modal sentence (\cite{bent:moda76corr}, see also \cite[Theorem 15.15]{bent:moda83} and \cite[\S 4]{gold:vari89}).

In \cite{gold:elem95,gold:elem93,gold:quas01} there is an analysis of quasi-modally definable classes, the varieties they generate, and the logics they determine.\footnote{In  \cite{gold:elem95,gold:elem93} I used the name `pseudo-equational' rather than `quasi-modal', but the latter seems more apposite in the modal context.} Results, for an arbitrary logic $L$, include:
\begin{itemize}
\item 
The canonical frames $\F_\lk$ for infinite $\k$ all satisfy the same quasi--modal sentences.
\item
If $\C_{qm}(L)$ is the class of all frames that satisfy the same quasi--modal sentences as the canonical frames $\F_\lk$, then the modal logic determined by $\C_{qm}(L)$ is the largest sublogic of $L$ that is first-order complete.
\item
If $ L$ is first-order complete, then it is determined by the quasi-modally definable elementary class $\C_{qm}(L)$.
\end{itemize}
Thus  if $L$ is determined by some first-order conditions, then it is determined by the quasi-modal conditions that are satisfied by the canonical frames $\F_\lk$.

\section{Canonical Varieties} \label{varieties}

We now review the algebraic semantics of modal logics (see \cite{gold:alge00} or \cite{blac:moda01}  for more details), and then derive a new characterisation of canonicity of a logic in Theorem \ref{canonicitychar}.

A \emph{modal algebra} $\A=(\B,l)$ consists of a Boolean algebra $\B$ with a function $l:\B\to\B$ that preserves the Boolean meet operation and the greatest element $1_\B$ of $\B$. Each modal formula $A$ may be viewed as a term in the language of 
$\A$, with the propositional variables of $A$ treated as variables ranging over the elements of $\A$, its Boolean connectives interpreted as the corresponding Boolean operations of $\B$, and $\Box$ interpreted as $l$.
If $A$ has $n$ variables, it induces an $n$-ary function on $\A$, and $A$ is said to be \emph{valid in} $\A$, written $\A\models A$,  if this term function takes the constant value $1_\B$. Thus $\A$ validates modal formula $A$ iff it satisfies the \emph{equation} ``$A=1$'' when $A$ is viewed as an algebraic term.

A frame $\F=(X,R)$ has the associated modal algebra $\Cm\F=(\P(X),l_R)$, called the \emph{complex algebra} of $\A$, where 
$\P(X)$ is the Boolean algebra of all subsets of $X$, and 
$$
l_R(Y)=\{a\in X: \forall b(aRb\text{ implies } b\in Y\}.
$$
Then in general, $\F\models A$ iff $\Cm\F\models A$.

In the converse direction, a modal algebra $\A$ has the \emph{canonical frame} $\Cf\A=(X_\B,R_l)$, where $X_\B$ is the set of ultrafilters of $\B$ and $FR_l G$ iff $\{a:la\in F\}\sub G$. 
Let $\Em\A=\Cm\Cf\A$, the complex algebra of the canonical frame $\Cf\A$. $\Em\A$ is the \emph{canonical embedding algebra} of $\A$.
There is an injective modal algebra homomorphism from $\A$ into $\Em\A$, given by the map $a\mapsto\{F\in X_B:a\in F\}$. This embedding of $\A$ into a complex algebra generalises the Stone representation of $\B$, and is due to J{\'o}nsson and Tarski \cite{jons:bool51}, who called $\Em\A$ the \emph{perfect extension} of $\A$.\footnote{The name \emph{canonical embedding algebra} and notation $\Em\A$ is from \cite[\S 2.7]{henk:cyli71}. 
Another common name is \emph{canonical extension}, with the notation $\A^\sig$ also being used.}

The transformations $\F\mapsto\Cm\F$ and $\A\mapsto\Cf\A$ give rise to a ``contravariant duality'' between frames and algebras. A bounded morphism $f:\F\to\F'$ induces a homomorphism $\theta_f:\Cm\F'\to\Cm\F$ in the reverse direction, taking each subset of $\F'$ to its $f$-inverse image. A homomorphism $\theta:\A\to\A'$ induces a bounded morphism
$f_\theta:\Cf\A'\to\Cf\A$ taking each ultrafilter $F$ of $\A'$ to the ultrafilter $\theta^{-1}F$ of $\A$. If $f$ is surjective or injective, then $\theta_f$ is injective or surjective, respectively. Likewise, $f_\theta$ interchanges surjectivity and injectivity with $\theta$. Composing the transformations show that any homorphism $\A\to\A'$ lifts to a homorphism
$\Em\A\to\Em\A'$ that preserves  surjectivity and injectivity.

A class $\V$  of modal algebras is a \emph{variety} if it the class of all models of some class of  equations. By the theorem of Birkhoff \cite{birk:stru35}, this holds iff $\V$ is closed under homomorphic images, subalgebras, and direct products. 
A variety $\V$ has \emph{free} algebras: for any cardinal $\k$ there is an algebra $\A_\V(\k)$ in $\V$ having a subset 
$\{a_\xi:\xi<\k\}$ that generates $\A_\V(\k)$, such that any mapping of this set of generators into any $\V$-algebra $\A$ extends to a homomorphism $\A_\V(\k)\to\A$. Thus any $\V$-algebra of size at most $\k$ is a homomorphic image of the free algebra $\A_\V(\k)$.

The $\Cf$ construction  is evidently an algebraic analogue of the construction of the canonical frame of a logic. The analogy can be made precise. For any modal logic $L$, the class $\V_L$ of algebras that validate all $L$-theorems is a variety, since modal formulas $A$ can be identified with modal-algebra equations $A=1$.
In $\V_L$ a free algebra on $\k$-many generators can be constructed as the standard  \emph{Lindenbaum-Tarski algebra} 
$\A_\lk$  of the $\k$-logic $\lk$.  The elements of this algebra are the equivalence classes $[A]$ of $\k$-formulas under provable material equivalence, i.e. $[A]=[B]$ iff $A\leftrightarrow B$ is an $L$-theorem.  $\A_\lk$ is an $L$-algebra and has the free generating set  $\{[p_\xi]: \xi<\k\}$.

Now if $\Delta$ is a maximally $\lk$-consistent set of $\k$-formulas, then $\{[A]:A\in\Delta\}$ is an ultrafilter of $\A_\lk$, and all ultrafilters of $\A_\lk$ have this form. This provides an isomorphism between the canonical $\lk$-frame $\F_\lk$ and the canonical frame $\Cf\A_\lk$ of the Lindenbaum-Tarski algebra. Hence $\Cm\F_\lk$ is isomorphic to
$\Cm\Cf\A_\lk=\Em\A_\lk$.
In summary:
\begin{quote}\em
The canonical frames $\F_\lk$ for a logic $L$ can be identified with the canonical frames of the infinitely generated free algebras in the variety $\V_L$ of all modal algebras that validate $L$. Moreover, the complex algebra $\Cm\F_\lk$ can be identified with the canonical embedding algebra of the free $\V_L$-algebra on $\k$-many generators.
\end{quote}
(See \cite[\S 5]{gold:alge00} for more detail.)

A variety $\V$ is called \emph{canonical} if it is closed under canonical embedding algebras,
i.e.\ $\Em\A$ belongs to $\V$ whenever $\A$ does. But any $\V$-algebra $\A$ is a homomorphic image of a free $\V$-algebra $\A_\V(\k)$ for some infinite $\k$, and so by duality  $\Em\A$ is a homomorphic image of  $\Em\A_\V(\k)$. Since $\V$ is closed under homomorphic images, this shows that

\begin{quote}
\emph{A variety is canonical if, and only if, it contains the canonical embedding algebras of all its free algebras on  infinitely many generators} \cite[Theorem 3.5.4]{gold:vari89}.
\end{quote}

In the case of the variety $\V_L$ of algebras validating a logic $L$, these observations combine to show that $\V_L$ is canonical iff it contains the complex algebras $\Cm\F_\lk$ for all infinite $\k$. But  $\Cm\F_\lk$ validates $L$ iff the frame $\F_\lk$ validates $L$, so in fact 
\begin{quote}\em
$\V_L$ is a canonical variety if, and only if, $L$ is a canonical logic.
\end{quote}

Our objective now is to give a new characterisation of canonicity of a logic $L$ by validity in the canonical frames 
$\Cf(\A_\lo^{\ \,U})$ of the ultrapowers of the denumerably generated free $L$-algebra. This depends on the relationship between
$\A_\lo$ and the other infinitely generated free algebras $\A_\lk$. Freeness itself gives a surjective homomorphism 
$\A_\lk\twoheadrightarrow\A_\lo$, but there is a deeper relationship:

\begin{lemma} \label{embedpower}
For any infinite cardinal $\k$ there is an injective homomorphism
$$
\theta:\A_\lk \rightarrowtail \A_\lo^{\ \,I}
$$
from $\A_\lk$ into some direct power of $\A_\lo$.
\end{lemma}
\begin{proof}
It suffices to show that  for any $a\ne 0$ in $\A_\lk$ there is a homomorphism $\theta_a:\A_\lk\to\A_\lo$ with $\theta_a(a)\ne 0$ in $\A_\lo$. For then, by taking $I$ as the set of non-zero elements of $\A_\lk$, the desired $\theta$ is given by the product map
$$
\theta(b)=\langle \theta_a(b):a\in I\rangle.
$$
So, suppose $a$ is the equivalence class $[A]$ of some $\k$-formula $A$ that has variables $p_{\xi_1},\dots p_{\xi_n}$ from $\L_\k$. Choose a list $q_1,\dots,q_n$ of distinct variables from $\L_\om$ that is disjoint from $p_{\xi_1},\dots p_{\xi_n}$.
Let $\sig$ be any function from $\L_\k$ into $\L_\om$ that has $\sig p_{\xi_i}=q_i$  for all $i\leq n$. Then $\sig$ extends to a substitution function $B\mapsto \sig B$ that maps $\k$-formulas to $\om$-formulas and commutes with the connectives.
Define  $\theta_a([B])=[\sig B]$ (the closure of $L_\k$ under substitution ensures this is well-defined). Then $\theta_a$ is a  homomorphism from $\A_\lk$ to $\A_\lo$.

Now since $[A]\ne 0$, $A$ is $\lk$-consistent, i.e.\ $\neg A$ is not an $\lk$-theorem. Therefore $\neg\sig A$ is not an $\lo$-theorem, or else $\neg A$ would be derivable in $\lk$ by substitution of $p_{\xi_i}$ for $q_i$ in $\neg\sig A$. Hence $\sig A$ is $\lo$-consistent, so $[\sig A]\ne 0$ in $\A_\lo$ as required.
\end{proof}

One other ingredient is required: a result about the canonical frame $\Cf(\Pi_I\A_i)$  of a direct product of algebras  whose dual form was shown by Gehrke and J\'onsson \cite[\S3.4]{gehr:boun04} to hold, not just for BAO's, but for distributive lattices with auxiliary operations. Here we give a proof  due to Ian Hodkinson \cite[Theorem 2.5]{gold:erdo04}, and restrict to the  case of a direct power $\A^I$, since that is all we need.

\begin{lemma} \label{ian}
For any modal algebra $\A$ and any set $I$,  the canonical frame $\Cf( \A^I)$ of the direct power $\A^I$ is isomorphic to the disjoint union of the canonical  frames 
$\Cf (\A^U)$ of the ultrapowers $\A^U$ as $U$ ranges over all ultrafilters on $I$.
\end{lemma}

\begin{proof}
For each ultrafilter $U$ on $I$, the quotient map $g\mapsto g_U$ is a homomorphism from $\A^I$ onto $\A^U$, inducing by duality an injective bounded morphism
$$
\ph_U:\Cf\A^U\rightarrowtail\Cf\A^I
$$
whose image $\Im\ph_U$ is an inner subframe of $\Cf\A^I$.

Now the frames $\Im \ph_U$ for all ultrafilters $U$ on $I$ form a family of pairwise disjoint inner subframes that covers $ \Cf\A^I $, i.e.\ $\Cf\A^I$ is their disjoint union. Since each $\Cf\A^U$ is isomorphic to $\Im\ph_U$, the Lemma follows from this.

For the details: if $J\sub I$, let $\chi_J\in \A^I$ be the characteristic function of $J$. We have:
\begin{equation}  \label{charJinF}
\text{if $J\in U$ and $F\in \Im \ph_U$,   then $\chi_J\in F$.}
\end{equation}
For, if $J\in U$, then $(\chi_J)_U=1$ in $\A^U$, so for any ultrafilter $G\in\Cf\A^U$, 
$(\chi_J)_U\in G$ and so  $\chi_J\in \ph_U(G)$.

This implies that $\Im \ph_U$ and $\Im \ph_{U'}$ are disjoint when $U\ne U'$. For in that case there is some $J\in U$ with $I-J\in U'$, and if there were some $F$ in both $\Im \ph_U$ and $\Im \ph_{U'}$,  result \eqref{charJinF} would imply that both $\chi_J$ and $\chi_{I-J}$ were in the ultrafilter $F$. But $\chi_{I-J}$ is the Boolean complement of $\chi_J$ in $\A^I$, so this is impossible.

Moreover, each $F\in\Cf\A^I$ belongs to $\Im \ph_U$ for some $U$, namely 
$U=\{J:\chi_J\in F\}$, which is an ultrafilter on $I$. The set $G=\{g_U: g\in F\}$ is readily seen to be an ultrafilter of $\A^U$, with the one part of this that depends on the particular definition of $U$ being that $G$ is proper, i.e.\ $0\notin G$. But if $g\in F$, let $J=\{i\in I: g(i)\ne 0\}$. Then $\chi_J$ dominates $g$, so also belongs to $F$, hence $J\in U$ and so $g_U\ne 0$ as required. Thus $G\in\Cf\A^U$, and as 
$F\sub\{g\in A^I: g_U\in G\}= \ph_U(G)$, maximality of $F$ ensures that $F=\ph_U(G)\in\Im \ph_U$.

This confirms that the frames $\Im \ph_U$  cover
$ \Cf\A^I $ as claimed.
\end{proof}

\begin{theorem} \label{canonicitychar}
For any logic $L$, the following are equivalent.
\begin{enumerate}[\rm(1)]
\item 
$L$ is canonical.
\item 
The frame $\Cf(\A_\lo^{\ \,U})$ validates $L$ for every ultrafilter $U$ on any set.
\item
The frame $\Cf(\A_\lo^{\ \,I})$ validates $L$ for every set $I$.
\end{enumerate}
\end{theorem}
\begin{proof}\strut

(1) implies (2): If $L$ is canonical, then its variety $\V_L$ is closed under $\Em$. Now $\A_\lo$ is in $\V_L$, hence so is any of its ultrapowers $\A_\lo^{\ \,U}$ (as ultrapowers preserve equations). Thus $\V_L$ contains
$\Em(\A_\lo^{\ \,U})=\Cm\Cf(\A_\lo^{\ \,U})$. So the algebra $\Cm\Cf(\A_\lo^{\ \,U})$ validates $L$, therefore so too does the frame $\Cf(\A_\lo^{\ \,U})$.

(2) implies (3):
Suppose $\Cf(\A_\lo^{\ \,U})\models L$ for all ultrafilters $U$. Then as disjoint unions and isomorphisms preserve validity, Lemma \ref{ian} implies that for any set $I$,  $\Cf( \A_\lo^{\ \,I})\models L$.

(3) implies (1):
For each infinite $\k$, by Lemma \ref{embedpower} there is a set $I$ and an injective homomorphism
$\A_\lk \rightarrowtail \A_\lo^{\ \,I}$. By duality, this induces a surjective bounded morphism
$\Cf(\A_\lo^{\ \,I})\twoheadrightarrow \Cf(\A_\lk)$. If (3) holds, then as validity is preserved by images of bounded morphisms, it follows that
$ \Cf(\A_\lk)\models L$. But $ \Cf(\A_\lk)$ is isomorphic to the canonical $\k$-frame $\F_\lk$, so this proves that $L$ is valid in all its canonical $\k$-frames, i.e.\ is canonical.
\end{proof}

Our results reveal a tight relationship between the notions `first-order complete' and `canonical'. The case $\k=\om$ of Theorem \ref{firstchar}, and the isomorphism of $\F_\lo$ with $\Cf\A_\lo$, show that
\begin{equation}  \label{focchar}
\textit{$L$ is first-order complete  iff  \,$(\Cf\A_\lo)^U\models L$  for all ultrafilters $U$;}
\end{equation}
while Theorem \ref{canonicitychar} just proved shows that
\begin{equation}  \label{conchar}
\textit{$L$ is canonical iff  \,$\Cf(\A_\lo^{\ \,U})\models L$  for all ultrafilters $U$.}
\end{equation}
So the difference between the two notions comes down to the fact that \emph{formation of canonical frames need not commute with ultrapowers}:  the frames $(\Cf\A_\lo)^U$ and $\Cf(\A_\lo^{\ \,U})$ need not be isomorphic. They need not even be of the same size:

\begin{theorem}\label{bigger}
If $U$ is a nonprincipal ultrafilter on some countable set $I$, then  $(\Cf\A_\lo)^U$ is of size $\cont$, the cardinal of the continuum, while  $\Cf(\A_\lo^{\ \,U})$  is of size $2^\cont$.
\end{theorem}
\begin{proof}
(In brief.)  The Lindenbaum-Tarski algebra $ \A_\lo$ is countable and atomless, so has exactly $\cont$ ultrafilters 
\cite[\S 1.7]{bell:mode69}. Thus  $\Cf\A_\lo$ is of size $\cont$. Since $\cont^\om=\cont$, it follows that $(\Cf\A_\lo)^U$ is also of size $\cont$ \cite[6.3.4]{bell:mode69}.

The set $P=\{[p_\xi]:\xi<\om\}$ of generators of $\A_\lo$ is \emph{independent}, meaning that if $X$ and $Y$ are any disjoint finite subsets of $P$, then $X\cup\{-a:a\in Y\}$ has non-zero meet. Within the ultrapower $\A_\lo^{\ \,U}$, the infinite  set $P$ gives rise to the `enlargement'
$$
^*P=\{f_U\in\A_\lo^{\ \,U}: \{i\in I:f(i)\in P\}\in U\}.
$$
$^*P$ is an independent subset of $\A_\lo^{\ \,U}$, as may be checked by standard ultrapower reasoning, or the fact that the structures   
$(\A_\lo,P)$ and $(\A_\lo^{\ \,U},{}^*P)$ are first-order equivalent. Hence for each $X\sub {}^*P$, the set $X\cup\{-a:a\in{}^*P -X\}$ has the finite intersection property, and so extends to an ultrafilter $F_X$ of $\A_\lo^{\ \,U}$. Then $X\ne Y$ implies $F_X\ne F_Y$, so this shows that  $\A_\lo^{\ \,U}$ has at least as many ultrafilters as there are subsets of $^*P$.

Now the assumptions on $U$ and $I$ imply that for any countably infinite set $Z$, the ultrapower $Z^U$ is of size $\cont$
 \cite[6.3.13]{bell:mode69}. So $P^U$ and $\A_\lo^{\ \,U}$ are of size $\cont$. The inclusion $P\hookrightarrow \A_\lo$ induces an injection $P^U\rightarrowtail \A_\lo^{\ \,U}$ whose image is $^*P$, so  $^*P$ is also of size $\cont$.
 Thus $\A_\lo^{\ \,U}$ has at least $2^\cont$ ultrafilters. Since it cannot have more, we conclude that $\Cf(\A_\lo^{\ \,U})$  is of size $2^\cont$.
 \end{proof}

The fact that there are logics $L$ that are canonical but not first-order complete means, by  \eqref{focchar} and \eqref{conchar}, that there are cases where $\Cf(\A_\lo^{\ \,U})\models L$ but $(\Cf\A_\lo)^U\not\models L$. In that case there is no structural relationship that preserves modal validity from  $\Cf(\A_\lo^{\ \,U})$ to $(\Cf\A_\lo)^U$.

In general there does exist an injective function
$$
(\Cf\A)^U \rightarrowtail  \Cf(\A^{U})
$$
making $(\Cf\A)^U$ isomorphic to a subframe of $\Cf(\A^{U})$, but not necessarily to an \emph{inner} subframe. The injection maps an element $g_U$ of $(\Cf\A)^U$ to the ultrafilter
$\{h_U \in \A^{U} : \{i\in  I : h(i) \in g(i)\} \in U\}$ of $\A^U$.

This has a bearing on the notion of a \emph{subframe logic}, introduced by Fine \cite{fine:logi85} and characterised as one whose validity is preserved by all subframes, not just the inner ones. He showed that if a subframe logic extending K4 is canonical, then it must be first-order complete. Wolter \cite{wolt:stru97} removed the restriction to logics with transitive frames, showing that every canonical subframe logic is first-order complete. That fact now follows immediately from our present observations, for if $L$ is a subframe logic, then the injection of $(\Cf\A_\lo)^U$ onto a subframe of $\Cf(\A_\lo^{\ \,U})$ ensures that if
$\Cf(\A_\lo^{\ \,U})\models L$ then $(\Cf\A_\lo)^U\models L$, which by \eqref{focchar} and \eqref{conchar} ensures  that if $L$ is canonical then it is also first-order complete.

\end{document}